\documentclass[11pt]{article}

\usepackage[margin=1.02in]{geometry}
\usepackage{amsmath,amssymb,amsthm,mathtools}
\usepackage{enumitem}
\usepackage{microtype}
\usepackage[T1]{fontenc}
\usepackage{lmodern}
\usepackage{booktabs}
\usepackage{array}
\usepackage{hyperref}
\usepackage[nameinlink,capitalise]{cleveref}

\hypersetup{colorlinks=true,linkcolor=blue,citecolor=blue,urlcolor=blue}
\setlist{itemsep=2pt,topsep=4pt}

\newtheorem{theorem}{Theorem}[section]
\newtheorem{proposition}[theorem]{Proposition}
\newtheorem{lemma}[theorem]{Lemma}
\newtheorem{corollary}[theorem]{Corollary}
\newtheorem{remark}[theorem]{Remark}

\theoremstyle{definition}

\crefname{theorem}{Theorem}{Theorems}
\Crefname{theorem}{Theorem}{Theorems}
\crefname{lemma}{Lemma}{Lemmas}
\Crefname{lemma}{Lemma}{Lemmas}
\crefname{proposition}{Proposition}{Propositions}
\Crefname{proposition}{Proposition}{Propositions}
\crefname{corollary}{Corollary}{Corollaries}
\Crefname{corollary}{Corollary}{Corollaries}
\crefname{remark}{Remark}{Remarks}
\Crefname{remark}{Remark}{Remarks}
\crefname{example}{Example}{Examples}
\Crefname{example}{Example}{Examples}

\newcommand{\Ass}{\operatorname{Ass}}
\newcommand{\Min}{\operatorname{Min}}
\newcommand{\Max}{\operatorname{Max}}
\newcommand{\Spec}{\operatorname{Spec}}
\newcommand{\Supp}{\operatorname{Supp}}
\newcommand{\Ann}{\operatorname{Ann}}
\newcommand{\rank}{\operatorname{rank}}
\newcommand{\len}{\operatorname{length}}
\newcommand{\Ext}{\operatorname{Ext}}
\newcommand{\Hom}{\operatorname{Hom}}
\newcommand{\Frac}{\operatorname{Frac}}

\title{ On the iso-Artinianness of one-dimensional Noetherian rings}
\author{Xiaolei Zhang,\ Ran Yan,\ Wei Qi\\
\small School of Mathematics and Statistics, Tianshui Normal University,
 Tianshui 741001, China\\
zxlrghj@163.com;\ yanrabc@163.com;\  qwrghj@126.com}
\date{}

\begin{document}
\maketitle

\begin{abstract}
Let $R$ be a one-dimensional commutative Noetherian ring with a unique minimal prime
$\mathfrak p$ such that $D=R/\mathfrak p$ is a principal ideal domain.  We give a
complete  characterization of the iso-Artinian property in this class, that is, the
following conditions are equivalent:
\begin{enumerate}[label=\textup{(\roman*)}]
	\item $R$ is iso-Artinian;
	\item $\mathfrak pR_{\mathfrak p}=0$;
	\item $\len_R(\mathfrak p)<\infty$;
	\item $\mathfrak p/\mathfrak p^2$  is torsion over $D$.
\end{enumerate} 

As a consequence, we completely answer the  Question~3.9 of Daneshvar and Divaani-Aazar:
under their hypotheses $\Min R\subsetneq\Ass R$, the ring is iso-Artinian precisely when its
nilradical has finite length, and it is non-iso-Artinian precisely when the conormal module has
positive rank. 
\end{abstract}



\section{Introduction}

Throughout the paper, all rings are commutative with identity and all modules are unital. About a decade ago, Facchini and Nazemian introduced iso-Artinian modules and rings as chain conditions in which
eventual equality is replaced by eventual isomorphism
\cite{FacchiniNazemian2016,FacchiniNazemian2017,FacchiniNazemian2019}.  A
commutative ring $R$ is \emph{iso-Artinian} if for every descending chain of ideals
\[
 I_1\supseteq I_2\supseteq I_3\supseteq\cdots
\]
there exists $n_0$ such that $I_n\cong I_{n_0}$ as $R$-modules for every $n\ge n_0$.
The condition is strictly weaker than the Artinian condition, for example, every non-field PID is
iso-Artinian, although it is not Artinian.

The structure of commutative iso-Artinian rings was studied in depth by Daneshvar and
Divaani-Aazar \cite{DaneshvarDivaaniAazar2025}.  Among other results, they proved that every
commutative iso-Artinian ring has Krull dimension at most one, every minimal prime is
associated, and the quotient by every associated prime is a PID.  Their decomposition theorem
isolates, as an important unresolved Noetherian case, one-dimensional rings with a unique
minimal prime.  They ended their paper with the following question
\cite[Question~3.9]{DaneshvarDivaaniAazar2025}:

\medskip
\noindent\emph{If $R$ is a one-dimensional Noetherian ring with a unique minimal prime
$\mathfrak p$, if $R/\mathfrak p$ is a PID, and if
$\Min R\subsetneq\Ass R$, must $R$ be iso-Artinian?}
\medskip

A natural question is how to characterize the iso-artinianess of  a one-dimensional Noetherian ring with a unique minimal prime? Our main result is the following complete characterization.  Standard facts on Noetherian
rings, finite-length modules, localization, and the Artin--Rees lemma are used throughout; see
\cite{AtiyahMacdonald1969,Eisenbud1995,Matsumura1986}.

\begin{theorem}[Main theorem]\label{thm:intro-main}
Let $R$ be a one-dimensional commutative Noetherian ring with unique minimal prime
$\mathfrak p$, and suppose that
\[
 D:=R/\mathfrak p
\]
is a principal ideal domain.  Then the following conditions are equivalent:
\begin{enumerate}[label=\textup{(\roman*)}]
\item $R$ is iso-Artinian;
\item $\mathfrak pR_{\mathfrak p}=0$;
\item $\mathfrak p$ has finite length as an $R$-module;
\item the $D$-module $\mathfrak p/\mathfrak p^2$ is torsion;
\item $\rank_D(\mathfrak p/\mathfrak p^2)=0$;
\item there exists $s\in R\setminus\mathfrak p$ such that $s\mathfrak p=0$.
\end{enumerate}
Consequently,
\[
 R\text{ is not iso-Artinian}
 \quad\Longleftrightarrow\quad
 \rank_D(\mathfrak p/\mathfrak p^2)>0.
\]
\end{theorem}

The result completely resolves the question above: the answer is positive exactly for
zero-dimensional nilpotent thickenings of the PID and negative exactly when the conormal
module $\mathfrak p/\mathfrak p^2$ has positive rank.  In particular, both positive and negative examples occur even when
$D$ is a DVR.

\section{Preliminaries}

 Let $R$ be a one-dimensional commutative Noetherian ring with unique minimal prime
 $\mathfrak p$.  Since $R$ is Noetherian, its nilradical is nilpotent.  Because
$\mathfrak p$ is the unique minimal prime,
\[
 \sqrt{0}=\mathfrak p.
\]
Moreover, the canonical map
\[
 \Spec R\longrightarrow\Spec D,
 \qquad
 \mathfrak q\longmapsto\mathfrak q/\mathfrak p,
\]
is a bijection.  Since $\dim R=1$, the PID $D$ is not a field.

We first record two standard facts that will be used repeatedly.

\begin{lemma}\label{lem:finite-length-support}
Let $M$ be a finitely generated module over a Noetherian ring $A$.  If
$\Supp_A(M)$ consists only of finitely many maximal ideals, then $M$ has finite length.
Conversely, a finite-length module is supported on finitely many maximal ideals.
\end{lemma}

\begin{proof}
If $\Supp_A(M)=\{\mathfrak m_1,\ldots,\mathfrak m_t\}$, then
$A/\Ann_A(M)$ is a zero-dimensional Noetherian ring, hence Artinian.  Since $M$ is finitely
generated over this Artinian ring, it has finite length.  The converse is standard: the support
of a finite-length module is the finite set of annihilators of its composition factors.
\end{proof}

\begin{lemma}\label{lem:ext-finite-length}
Let $A$ be Noetherian, let $\mathfrak q\in\Spec A$, put $B=A/\mathfrak q$, and let $J$ be
a finite-length $A$-module.  Then $\Ext^1_A(B,J)$ is a finite-length $A$-module.  Moreover,
$\mathfrak q$ annihilates $\Ext^1_A(B,J)$, so it is naturally a $B$-module.
\end{lemma}

\begin{proof}
Because $A$ is Noetherian and $B=A/\mathfrak q$ is finitely generated, there is an exact
sequence
\[
 0\longrightarrow K\longrightarrow F\longrightarrow B\longrightarrow0
\]
with $F$ a finite free $A$-module and $K$ finitely generated.  Applying
$\Hom_A(-,J)$ gives a surjection
\[
 \Hom_A(K,J)\twoheadrightarrow \Ext^1_A(B,J).
\]
The module $\Hom_A(K,J)$ is finitely generated and
\[
 \Supp_A\Hom_A(K,J)\subseteq\Supp_A(J).
\]
Consequently $\Ext^1_A(B,J)$ is finitely generated and its support is also contained in
$\Supp_A(J)$.  Since $J$ has finite length, this support is a finite set of maximal ideals;
Lemma~\ref{lem:finite-length-support} therefore shows that $\Ext^1_A(B,J)$ has finite length.

For $r\in\mathfrak q$, multiplication by $r$ on $B$ is the zero endomorphism.  Functoriality
of Ext in the first variable therefore shows that $r$ acts trivially on
$\Ext^1_A(B,J)$.  For a commutative ring the scalar actions obtained from the first and second
variables coincide, so $\mathfrak q\Ext^1_A(B,J)=0$.  Thus $\Ext^1_A(B,J)$ is naturally a
$B$-module.
\end{proof}

We shall also use the following elementary observation concerning scalar automorphisms of a
finite-length module.

\begin{lemma}\label{lem:scalar-auto}
Let $A$ be Noetherian and let $J$ be a finite-length $A$-module.  If
$r\notin\mathfrak m$ for every $\mathfrak m\in\Supp_A(J)$, then multiplication by $r$ is an
automorphism of $J$.
\end{lemma}

\begin{proof}
After localizing at a prime $\mathfrak q$, either $J_{\mathfrak q}=0$, or
$\mathfrak q\in\Supp_A(J)$ and $r$ becomes a unit.  Hence multiplication by $r$ is an
isomorphism after localization at every prime, and therefore is an isomorphism globally.
\end{proof}

The following proposition packages the intrinsic conditions which will occur in the main
theorem.

\begin{proposition}\label{prop:intrinsic-equivalences}
Let $R$ be a one-dimensional Noetherian ring with unique minimal prime $\mathfrak p$ and let
$D=R/\mathfrak p$ be a PID.  Then the following conditions are equivalent:
\begin{enumerate}[label=\textup{(\roman*)}]
\item $\mathfrak pR_{\mathfrak p}=0$;
\item $\mathfrak p$ has finite length;
\item there exists $s\in R\setminus\mathfrak p$ with $s\mathfrak p=0$;
\item $\mathfrak p/\mathfrak p^2$ is a torsion $D$-module;
\item $\rank_D(\mathfrak p/\mathfrak p^2)=0$.
\end{enumerate}
\end{proposition}

\begin{proof}
$(i)\Rightarrow (iii):$  Since $\mathfrak p$ is finitely generated, choose generators
$x_1,\ldots,x_t$.  For each $i$ there is $s_i\notin\mathfrak p$ with $s_ix_i=0$.
Then $s=s_1\cdots s_t\notin\mathfrak p$ and $s\mathfrak p=0$.

$(iii)\Rightarrow (ii):$  Every prime ideal containing $s$ is different from the unique minimal prime
$\mathfrak p$ and is therefore maximal.  Hence $R/(s)$ is a zero-dimensional Noetherian ring,
so it is Artinian.  The finitely generated $R/(s)$-module $\mathfrak p$ has finite length. 

$(ii)\Rightarrow (i):$ If $\mathfrak p$ has finite length, then its localization at the non-maximal prime
$\mathfrak p$ is zero. 

$(i)\Leftrightarrow (iv):$  Let $K=\Frac(D)$.  Since localization at
$\mathfrak p$ corresponds to localization of $D$ at its zero prime,
\[
 (\mathfrak p/\mathfrak p^2)\otimes_DK
 \cong
 \mathfrak pR_{\mathfrak p}/\mathfrak p^2R_{\mathfrak p}.
\]
If $\mathfrak pR_{\mathfrak p}=0$, the right-hand side vanishes.  Conversely, if the
right-hand side vanishes, then
\[
 \mathfrak pR_{\mathfrak p}=\mathfrak p^2R_{\mathfrak p}.
\]
The ring $R_{\mathfrak p}$ is a zero-dimensional Noetherian local ring whose maximal ideal is
$\mathfrak pR_{\mathfrak p}$.  Nakayama's lemma applied to the finitely generated module
$\mathfrak pR_{\mathfrak p}$ gives $\mathfrak pR_{\mathfrak p}=0$. 

$(iv)\Leftrightarrow (v):$ Finally, $\mathfrak p/\mathfrak p^2$ is finitely generated over the domain $D$, so
being torsion is equivalent to having rank zero. 
\end{proof}

\begin{remark}\label{rem:geometry}
Condition~(i) of Proposition~\ref{prop:intrinsic-equivalences} says that the generic local ring of $R$ is the
field $K=\Frac(D)$.  Hence it is precisely the condition that the one-dimensional scheme
$\Spec R$ be generically reduced.  Condition (ii) says equivalently that the nilpotent
structure is supported on finitely many closed points.
\end{remark}

\section{A stabilization theorem for finite-length extensions}

The positive direction of the main theorem is a consequence of the next result.  It is useful in
its own right because it does not require the nilradical to be square-zero or the extension to
split.

\begin{theorem}\label{thm:extension-stabilization}
Let $R$ be Noetherian with unique minimal prime $\mathfrak p$, and assume that
$D=R/\mathfrak p$ is a PID.  Let
\[
 I_1\supseteq I_2\supseteq I_3\supseteq\cdots
\]
be a descending chain of ideals.  Suppose that for some $n_0$ there is a finite-length
$R$-module $J$ such that
\[
 I_n\cap\mathfrak p=J\qquad(n\ge n_0).
\]
Then the chain $(I_n)_{n\ge n_0}$ is eventually constant up to isomorphism.
\end{theorem}

\begin{proof}
Put
\[
 A_n=(I_n+\mathfrak p)/\mathfrak p\subseteq D.
\]
If $A_n=0$ for some $n\ge n_0$, then $I_k\subseteq\mathfrak p$ for all $k\ge n$; hence
$I_k=J$ for all $k\ge n$, and there is nothing to prove.  We may therefore assume that every
$A_n$ is nonzero.

Since $D$ is a PID, choose generators $a_n\in D$ with $A_n=a_nD$ in such a way that
\[
 a_{n+1}=a_nd_n\in A_{n+1}\qquad(d_n\in D),
\]
as $A_n\supseteq A_{n+1}.$
For every $n\ge n_0$ we have a short exact sequence of $R$-modules
\begin{equation}\label{eq:extension-In}
 0\longrightarrow J\longrightarrow I_n\longrightarrow A_n\longrightarrow0.
\end{equation}
Using the $R$-isomorphism $D\to A_n$, $x\mapsto a_nx$, let
\[
 e_n\in E:=\Ext^1_R(D,J)
\]
be the Yoneda class of \eqref{eq:extension-In}.

We claim that
\begin{equation}\label{eq:recurrence-ext}
 e_{n+1}=d_ne_n.
\end{equation}
Indeed, the inverse image of $A_{n+1}$ under the map $I_n\to A_n$ is exactly $I_{n+1}$.
To see this, if $x\in I_n$ maps into $A_{n+1}$, choose $y\in I_{n+1}$ with the same image.
Then $x-y\in I_n\cap\mathfrak p=J\subseteq I_{n+1}$, so $x\in I_{n+1}$.  Thus the
extension for $I_{n+1}$ is the pullback of the extension for $I_n$ along multiplication by
$d_n$ on $D$.  In the standard $D$-module structure on Ext, this is exactly
\eqref{eq:recurrence-ext}.

By Lemma~\ref{lem:ext-finite-length}, $E$ is a finite-length $D$-module.  Hence the descending
chain of cyclic submodules
\[
 De_{n_0}\supseteq De_{n_0+1}\supseteq De_{n_0+2}\supseteq\cdots
\]
stabilizes.  Choose $N$ such that
\[
 De_n=De_N\qquad(n\ge N).
\]
We prove that $I_n\cong I_m$ for all $m\ge n\ge N$.

Fix such $m,n$.  By repeated use of \eqref{eq:recurrence-ext}, there is $d\in D$ with
\[
 e_m=de_n.
\]
If $e_n=0$, then also $e_m=0$, and both extensions split; hence
$I_n\cong J\oplus D\cong I_m$.  Assume $e_n\ne0$, and put
\[
 B=\Ann_D(e_n).
\]
Since $De_m=De_n$, the element $de_n$ generates the cyclic module $De_n\cong D/B$.
Therefore
\begin{equation}\label{eq:d-unit-mod-B}
 dD+B=D.
\end{equation}

Let
\[
 \Supp_R(J)=\{\mathfrak m_1,\ldots,\mathfrak m_t\}
\]
and write $\mathfrak q_i=\mathfrak m_i/\mathfrak p\in\Max D$; this is legitimate
because the unique minimal prime $\mathfrak p$ is contained in every prime of $R$.
Identify primes of $D$ with primes of $R$ containing $\mathfrak p$.  Since
$De_n$ is a submodule of the finite-length $D$-module $E$ and
$\Supp_R(E)\subseteq\Supp_R(J)$, every prime divisor of $B$ is among the
$\mathfrak q_i$.

If $B=0$, then \eqref{eq:d-unit-mod-B} says that $d$ is a unit of $D$; in this case set
$c=d$.  Assume $B\ne0$.  Since $De_n\cong D/B$ and
$\Supp_D(De_n)\subseteq\Supp_D(E)$, every maximal ideal of $D$
containing $B$ is one of the $\mathfrak q_i$.  Moreover,
\eqref{eq:d-unit-mod-B} implies that $d$ belongs to none of the maximal ideals containing $B$.
For every $\mathfrak q_i$ not containing $B$, we have $B+\mathfrak q_i=D$.  Hence the Chinese
remainder theorem gives $c\in D$ satisfying
\begin{equation}\label{eq:c-choice}
 c\equiv d\pmod B,
 \qquad
 c\equiv1\pmod{\mathfrak q_i}
 \quad\text{for every }\mathfrak q_i\not\supseteq B.
\end{equation}
If $\mathfrak q_i\supseteq B$, then $c-d\in B\subseteq\mathfrak q_i$ and
$d\notin\mathfrak q_i$, so $c\notin\mathfrak q_i$.  If $\mathfrak q_i\not\supseteq B$,
then $c\equiv1\pmod{\mathfrak q_i}$, so again $c\notin\mathfrak q_i$.  Thus $c$ avoids every
prime in $\Supp_R(J)$, and because $c-d\in B=\Ann_D(e_n)$ we have $ce_n=de_n$.

Choose a lift $r\in R$ of $c$.  Then $r\notin\mathfrak m_i$ for every
$\mathfrak m_i\in\Supp_R(J)$, so multiplication by $r$ is an automorphism of $J$ by
Lemma~\ref{lem:scalar-auto}.  Push out the extension representing $e_n$ along this automorphism.
Because the two scalar actions on Ext agree for commutative rings, the new extension class is
\[
 re_n=ce_n=de_n=e_m.
\]
Pushing out along an automorphism of the kernel does not change the isomorphism type of the
middle module.  Since two Yoneda extensions representing the same class are equivalent, the
middle terms are isomorphic.  Hence $I_n\cong I_m$.

Thus all ideals in the tail beginning at $N$ are mutually isomorphic, as required.
\end{proof}

\begin{corollary}\label{cor:sufficiency}
Under the hypotheses of \cref{thm:intro-main}, if $\mathfrak p$ has finite length, then $R$ is
iso-Artinian.
\end{corollary}

\begin{proof}
Let $I_1\supseteq I_2\supseteq\cdots$ be a descending chain of ideals.  The modules
$I_n\cap\mathfrak p$ form a descending chain of submodules of the finite-length module
$\mathfrak p$, so they eventually stabilize.  Apply \cref{thm:extension-stabilization}.
\end{proof}

\section{Generic nilpotence produces a non-isomorphic ideal chain}

We now prove the converse.  The argument is constructive and is the main technical point of
the paper.

\begin{theorem}\label{thm:generic-obstruction}
Let $R$ be a one-dimensional Noetherian ring with unique minimal prime $\mathfrak p$, and
suppose that $D=R/\mathfrak p$ is a PID.  If
\[
 \mathfrak pR_{\mathfrak p}\ne0,
\]
then $R$ is not iso-Artinian.  More precisely, $R$ contains a descending chain of ideals whose
members are pairwise non-isomorphic eventually.
\end{theorem}

\begin{proof}
Set
\begin{equation}\label{eq:Qdef}
 Q:=\mathfrak p\cap(0:_R\mathfrak p).
\end{equation}
Then $Q$ is a finitely generated ideal, $\mathfrak pQ=0$, and therefore $Q$ is naturally a
finitely generated $D$-module.

We first show that $Q$ has positive $D$-rank.  Put
\[
 A=R_{\mathfrak p},\qquad \mathfrak n=\mathfrak pR_{\mathfrak p}.
\]
The ring $A$ is zero-dimensional Noetherian local, hence Artinian, and by assumption
$\mathfrak n\ne0$.  Choose $t\ge2$ minimal with $\mathfrak n^t=0$.  Then
$0\ne\mathfrak n^{t-1}\subseteq\mathfrak n\cap(0:_A\mathfrak n)$.  Since localization
commutes with the annihilator of the finitely generated ideal $\mathfrak p$,
\[
 Q_{\mathfrak p}
 =\mathfrak n\cap(0:_A\mathfrak n)\ne0.
\]
As $Q$ is a $D$-module and localization at $\mathfrak p$ corresponds to tensoring with
$K=\Frac(D)$, we have $Q\otimes_DK\ne0$.  Hence
\[
 r:=\rank_DQ>0.
\]
Write, using the structure theorem over a PID,
\begin{equation}\label{eq:Qdecomp}
 Q\cong D^r\oplus T.
\end{equation}

Choose a prime element $\pi\in D$ and a lift $a\in R$.  Decompose the torsion part as
\[
 T=T_\pi\oplus T',
\]
where $T_\pi$ is the $\pi$-primary component and $T'$ is the sum of the primary components
corresponding to primes non-associate to $\pi$.  Choose $e\ge1$ such that
$\pi^eT_\pi=0$.  Multiplication by $\pi$ is an automorphism of $T'$.

For $n\ge e$, define the $D$-submodule
\begin{equation}\label{eq:Ln}
 L_n:=\pi^nD^r\oplus T'\subseteq Q
\end{equation}
and the ideal
\begin{equation}\label{eq:Jn-general}
 J_n:=a^{2n}R+L_n.
\end{equation}
Because $\mathfrak pQ=0$, every $D$-submodule of $Q$ is an $R$-submodule, hence an ideal.
Thus $J_n$ is an ideal and
\[
 J_e\supseteq J_{e+1}\supseteq J_{e+2}\supseteq\cdots.
\]

We shall distinguish these ideals by the $D$-modules
\[
 V_n:=J_n/\mathfrak pJ_n.
\]
Since $\mathfrak pL_n=0$,
\begin{equation}\label{eq:pJn}
 \mathfrak pJ_n=a^{2n}\mathfrak p.
\end{equation}
Moreover,
\begin{equation}\label{eq:intersection-basic}
 a^{2n}R\cap\mathfrak p=a^{2n}\mathfrak p.
\end{equation}
Indeed, if $a^{2n}x\in\mathfrak p$, then in the domain $D$ we have
$\pi^{2n}(x+\mathfrak p)=0$, so $x\in\mathfrak p$.

Because $L_n\subseteq\mathfrak p$, we also have
\[
 J_n\cap\mathfrak p=a^{2n}\mathfrak p+L_n.
\]
Indeed, if $a^{2n}x+\ell\in\mathfrak p$ with $\ell\in L_n$, then
$a^{2n}x\in\mathfrak p$, so $x\in\mathfrak p$ by \eqref{eq:intersection-basic}.
Projection modulo $\mathfrak p$ therefore yields an exact sequence of $D$-modules
\begin{equation}\label{eq:Vn-exact}
 0\longrightarrow K_n\longrightarrow V_n\longrightarrow \pi^{2n}D\longrightarrow0,
\end{equation}
where
\begin{equation}\label{eq:Kn}
 K_n
 \cong
 \frac{a^{2n}\mathfrak p+L_n}{a^{2n}\mathfrak p}
 \cong
 \frac{L_n}{L_n\cap a^{2n}\mathfrak p}.
\end{equation}
Since $\pi^{2n}D\cong D$ is projective, \eqref{eq:Vn-exact} splits:
\begin{equation}\label{eq:Vn-split}
 V_n\cong D\oplus K_n.
\end{equation}
The module $K_n$ is torsion.  Indeed, for $n\ge e$,
\[
 a^{2n}Q
 =\pi^{2n}D^r\oplus T'
 \subseteq L_n\cap a^{2n}\mathfrak p,
\]
and $a^{2n}Q$ has the same $D$-rank $r$ as $L_n$.

It remains to prove that the torsion modules $K_n$ are eventually pairwise non-isomorphic.
Let
\[
 \mathfrak m:=\{x\in R:x+\mathfrak p\in\pi D\},
 \qquad
 V:=D_{(\pi)}.
\]
Thus $V$ is a DVR with uniformizer $\pi$.  Apply the Artin--Rees lemma to the inclusion
$Q\subseteq\mathfrak p$ and the ideal $(a)$.  There is $c\ge0$ such that, for every $N\ge c$,
\begin{equation}\label{eq:AR}
 Q\cap a^N\mathfrak p
 =a^{N-c}(Q\cap a^c\mathfrak p).
\end{equation}
Set $H=Q\cap a^c\mathfrak p$ and localize at $\mathfrak m$.  In view of
\eqref{eq:Qdecomp},
\[
 Q_{\mathfrak m}\cong V^r\oplus (T_\pi)_{(\pi)}.
\]
Let $\overline H$ be the image of $H_{\mathfrak m}$ in the free quotient $V^r$.  Since
$a^cQ\subseteq H$, the lattice $\overline H$ has full rank $r$.  By the elementary divisor
theorem over the DVR $V$, after changing a basis of $V^r$ there exist integers
$b_1,\ldots,b_r\ge0$ such that
\begin{equation}\label{eq:H-lattice}
 \overline H=\pi^{b_1}Ve_1\oplus\cdots\oplus\pi^{b_r}Ve_r.
\end{equation}

Let $W=(T_\pi)_{(\pi)}$, the torsion summand of
$Q_{\mathfrak m}\cong V^r\oplus W$.  Since $W$ has finite length over the DVR $V$,
there is $u\ge 0$ with $\pi^uW=0$.  Choose $n$ so large that $2n-c\ge u$.
For every $h=(v,w)\in H_{\mathfrak m}\subseteq V^r\oplus W$ we then have
$\pi^{2n-c}h=(\pi^{2n-c}v,0)$.  Consequently
\[
 \pi^{2n-c}H_{\mathfrak m}=\pi^{2n-c}\overline H
 =\bigoplus_{i=1}^r\pi^{2n-c+b_i}Ve_i.
\]
Using \eqref{eq:AR}, this gives
\[
 (Q\cap a^{2n}\mathfrak p)_{\mathfrak m}
 =\bigoplus_{i=1}^r\pi^{2n-c+b_i}Ve_i.
\]
On the other hand, $T'$ vanishes after localization at $(\pi)$, and hence
\[
 (L_n)_{\mathfrak m}=\pi^nV^r.
\]
For all sufficiently large $n$ we have $2n-c+b_i\ge n$ for every $i$.  Hence
$(Q\cap a^{2n}\mathfrak p)_{\mathfrak m}\subseteq(L_n)_{\mathfrak m}$, and therefore
\[
 (L_n\cap a^{2n}\mathfrak p)_{\mathfrak m}
 =(Q\cap a^{2n}\mathfrak p)_{\mathfrak m}.
\]
Consequently
\begin{align}
 (K_n)_{(\pi)}
 &\cong
 \frac{\pi^nV^r}
 {\bigoplus_{i=1}^r\pi^{2n-c+b_i}Ve_i},\label{eq:Kn-local}\\
 \len_V((K_n)_{(\pi)})
 &=rn+\sum_{i=1}^r(b_i-c).\label{eq:length-growth}
\end{align}
Since $r>0$, the length in \eqref{eq:length-growth} is strictly increasing with $n$.

Finally, suppose $J_n\cong J_m$ as $R$-modules.  Every $R$-isomorphism sends
$\mathfrak pJ_n$ onto $\mathfrak pJ_m$, so it induces a $D$-isomorphism
$V_n\cong V_m$.  By \eqref{eq:Vn-split}, the torsion submodule of $V_n$ is exactly $K_n$.
Therefore $K_n\cong K_m$, contradicting \eqref{eq:length-growth} when $n\ne m$ are
sufficiently large.  Thus the chain \eqref{eq:Jn-general} does not terminate up to isomorphism,
and $R$ is not iso-Artinian.
\end{proof}

\section{The complete characterization}

We can now prove the main theorem announced in the introduction.

\begin{theorem}\label{thm:complete}
Let $R$ be a one-dimensional Noetherian ring with unique minimal prime $\mathfrak p$, and
assume that $D=R/\mathfrak p$ is a PID.  Then the following are equivalent:
\begin{enumerate}[label=\textup{(\roman*)}]
\item $R$ is iso-Artinian;
\item $R$ is generically reduced, i.e. $R_{\mathfrak p}$ is a field;
\item $\mathfrak pR_{\mathfrak p}=0$;
\item $\mathfrak p$ has finite length;
\item $\mathfrak p/\mathfrak p^2$ is a torsion $D$-module;
\item $\rank_D(\mathfrak p/\mathfrak p^2)=0$;
\item there exists $s\in R\setminus\mathfrak p$ with $s\mathfrak p=0$.
\end{enumerate}
\end{theorem}

\begin{proof}
The equivalence of (ii) and (iii) is immediate because $R_{\mathfrak p}$ is local with maximal
ideal $\mathfrak pR_{\mathfrak p}$ and residue field $\Frac(D)$.  The equivalence of
(iii)--(vii) is Proposition~\ref{prop:intrinsic-equivalences}.  Condition (iv) implies (i) by
Corollary~\ref{cor:sufficiency}, while the contrapositive of (i)$\Rightarrow$(iii) is
\cref{thm:generic-obstruction}.
\end{proof}

\begin{corollary}[Exact negative criterion]\label{cor:negative}
Under the hypotheses of \cref{thm:complete},
\[
 R\text{ is not iso-Artinian}
 \quad\Longleftrightarrow\quad
 \rank_D(\mathfrak p/\mathfrak p^2)>0.
\]
Whenever this rank is positive, the proof of \cref{thm:generic-obstruction} gives an explicit
descending chain of eventually pairwise non-isomorphic ideals.
\end{corollary}

\begin{remark}\label{rem:conormal}
The conormal module $\mathfrak p/\mathfrak p^2$ is therefore the exact general replacement for
the module occurring in a square-zero idealization.  No splitting of $R\to D$ is required.
\end{remark}

\section{Daneshvar and Divaani-Aazar's Question}

We now specialize the characterization to the setting that motivated the problem.

\begin{theorem}[Complete answer to Question~3.9]\label{thm:q39}
Let $R$ be a one-dimensional Noetherian ring with unique minimal prime $\mathfrak p$, suppose
that $D=R/\mathfrak p$ is a PID, and assume
\[
 \Min R\subsetneq\Ass R.
\]
Then
\[
 R\text{ is iso-Artinian}
 \quad\Longleftrightarrow\quad
 \len_R(\mathfrak p)<\infty
 \quad\Longleftrightarrow\quad
 \rank_D(\mathfrak p/\mathfrak p^2)=0.
\]
Equivalently, $R$ is non-iso-Artinian exactly when
$\rank_D(\mathfrak p/\mathfrak p^2)>0$.
\end{theorem}

\begin{proof}
This is the specialization of \cref{thm:complete}.  Notice that the strict inclusion of
associated primes is not needed for the equivalence itself; it identifies precisely the case
asked about in \cite[Question~3.9]{DaneshvarDivaaniAazar2025}.
\end{proof}

In the positive case the associated primes admit a particularly clean description.

\begin{proposition}\label{prop:ass-positive}
Assume the hypotheses of \cref{thm:complete} and suppose that $R$ is iso-Artinian.  Then
\[
 \Ass_RR=\{\mathfrak p\}\cup\Supp_R(\mathfrak p).
\]
In particular, if $\mathfrak p\ne0$, then
$\Min R\subsetneq\Ass R$ automatically.
\end{proposition}

\begin{proof}
By \cref{thm:complete}, $\mathfrak p$ has finite length.  Since $R$ is Noetherian, the minimal
prime $\mathfrak p$ belongs to $\Ass_RR$.

Let $\mathfrak m\in\Supp_R(\mathfrak p)$.  The finite-length module $\mathfrak p$ has
zero-dimensional support, so every member of its support is minimal in that support and hence
belongs to $\Ass_R(\mathfrak p)$.  Since $\mathfrak p$ is a submodule of $R$,
$\Ass_R(\mathfrak p)\subseteq\Ass_RR$.  This proves
\[
 \{\mathfrak p\}\cup\Supp_R(\mathfrak p)\subseteq\Ass_RR.
\]

Conversely, let $\mathfrak m\in\Ass_RR$ with $\mathfrak m\ne\mathfrak p$.  Then
$\mathfrak m$ is maximal.  If $\mathfrak m\notin\Supp_R(\mathfrak p)$, then
$\mathfrak pR_{\mathfrak m}=0$, so
\[
 R_{\mathfrak m}\cong D_{\mathfrak m/\mathfrak p}
\]
is a DVR.  Its only associated prime is zero.  But localization of associated primes would give
$\mathfrak mR_{\mathfrak m}\in\Ass(R_{\mathfrak m})$, a contradiction.  Hence
$\mathfrak m\in\Supp_R(\mathfrak p)$.

If $\mathfrak p\ne0$, its finite-length support is nonempty, so there is at least one embedded
maximal associated prime.  Thus $\Min R\subsetneq\Ass R$.
\end{proof}

\begin{corollary}\label{cor:q39-negative-exists}
 Daneshvar and Divaani-Aazar's Question~3.9 has a negative answer  among square-zero
extensions over DVRs.
\end{corollary}

\begin{proof}
Let $V$ be a DVR with uniformizer $\pi$ and set
\[
 M=V\oplus V/(\pi),
 \qquad
 R=V\ltimes M.
\]
Then $R$ is a one-dimensional Noetherian ring with unique minimal prime
$\mathfrak p=0\ltimes M$ and $R/\mathfrak p\cong V$.  The element
$z=(0,(0,\overline1))\in R$ satisfies
\[
 \Ann_R(z)=(\pi)\ltimes M,
\]
so $\Min R\subsetneq\Ass R$.  On the other hand,
$\mathfrak p^2=0$ and
\[
 \mathfrak p/\mathfrak p^2\cong M
\]
has $V$-rank one.  Hence \cref{thm:complete} shows that $R$ is not iso-Artinian.

For completeness, the failure can also be seen directly.  For $n\ge1$ put
\[
 J_n=\pi^{2n}V\ltimes(\pi^nV\oplus0).
\]
These form a descending chain of ideals.  Since
$(0\ltimes M)J_n=0\ltimes(\pi^{2n}V\oplus0)$,
\[
 J_n/(0\ltimes M)J_n\cong V\oplus V/(\pi^n).
\]
The torsion submodule on the right has annihilator $(\pi^n)$, so these quotients, and hence the
ideals $J_n$, are pairwise non-isomorphic.  Thus $R$ is not iso-Artinian.
\end{proof}

\section{Applications}

The complete characterization has several useful forms.

\begin{corollary}\label{cor:dvr-local}
Let $(R,\mathfrak m)$ be a one-dimensional Noetherian local ring with unique minimal prime
$\mathfrak p$.  Assume that $D=R/\mathfrak p$ is a DVR, let $\pi$ be a uniformizer of $D$,
and choose $x\in\mathfrak m$ lifting $\pi$.  Then the following are equivalent:
\begin{enumerate}[label=\textup{(\roman*)}]
\item $R$ is iso-Artinian;
\item $\mathfrak p$ has finite length;
\item $x^N\mathfrak p=0$ for some $N\ge1$;
\item $\mathfrak p/\mathfrak p^2$ is a torsion $D$-module.
\end{enumerate}
\end{corollary}

\begin{proof}
The equivalence of (i), (ii), and (iv) follows from \cref{thm:complete}.  If $\mathfrak p$ has
finite length, then a sufficiently large power of $\mathfrak m$ annihilates it, hence so does a
power of $x$.  Conversely, $x\notin\mathfrak p$, so $x^N\mathfrak p=0$ implies condition~(vii) of
\cref{thm:complete}.
\end{proof}

\begin{corollary}\label{cor:local-global}
Under the hypotheses of \cref{thm:complete}, the following are equivalent:
\begin{enumerate}[label=\textup{(\roman*)}]
\item $R$ is iso-Artinian;
\item $R_{\mathfrak m}$ is iso-Artinian for every maximal ideal $\mathfrak m$;
\item $R_{\mathfrak m}$ is iso-Artinian for at least one maximal ideal $\mathfrak m$.
\end{enumerate}
Thus, within this class, either every maximal localization is iso-Artinian or none is.
\end{corollary}

\begin{proof}
If $R$ is iso-Artinian, then $\mathfrak pR_{\mathfrak p}=0$ by \cref{thm:complete}.  For every
maximal $\mathfrak m$, the localized nilradical $\mathfrak pR_{\mathfrak m}$ has zero
localization at its minimal prime, so \cref{thm:complete} applied to $R_{\mathfrak m}$ shows
that $R_{\mathfrak m}$ is iso-Artinian.  Hence (i) implies (ii), and (ii) implies (iii).

Conversely, suppose $R_{\mathfrak m}$ is iso-Artinian for one maximal ideal
$\mathfrak m$.  Its unique minimal prime is $\mathfrak pR_{\mathfrak m}$ and
\[
 R_{\mathfrak m}/\mathfrak pR_{\mathfrak m}
 \cong D_{\mathfrak m/\mathfrak p},
\]
which is a DVR.  Applying \cref{thm:complete} to $R_{\mathfrak m}$ gives
\[
 (\mathfrak pR_{\mathfrak m})
 (R_{\mathfrak m})_{\mathfrak pR_{\mathfrak m}}=0.
\]
But
\[
 (R_{\mathfrak m})_{\mathfrak pR_{\mathfrak m}}\cong R_{\mathfrak p},
\]
so $\mathfrak pR_{\mathfrak p}=0$.  Hence $R$ is iso-Artinian by
\cref{thm:complete}.
\end{proof}

\begin{corollary}\label{cor:square-zero}
Let $R$ be a one-dimensional Noetherian ring with unique minimal prime $\mathfrak p$, assume
$\mathfrak p^2=0$, and suppose $D=R/\mathfrak p$ is a PID.  Then 
\[
 R\text{ is iso-Artinian}
 \quad\Longleftrightarrow\quad
 \mathfrak p\text{ is a torsion }D\text{-module}.
\]
No splitting of the square-zero extension $0\to\mathfrak p\to R\to D\to0$ is required.
\end{corollary}

\begin{proof}
Because $\mathfrak p^2=0$, the $R$-action on $\mathfrak p$ factors through $D$, and
$\mathfrak p/\mathfrak p^2=\mathfrak p$.  Apply \cref{thm:complete}.
\end{proof}

\begin{corollary}\label{cor:idealization}
Let $D$ be a non-field PID and let $M$ be a finitely generated $D$-module.  For the
idealization $R=D\ltimes M$,
\[
 R\text{ is iso-Artinian}
 \quad\Longleftrightarrow\quad
 M\text{ is torsion over }D.
\]
\end{corollary}

\begin{proof}
The unique minimal prime is $\mathfrak p=0\ltimes M$, with $\mathfrak p^2=0$ and
$R/\mathfrak p\cong D$.  The conormal module is naturally
$\mathfrak p/\mathfrak p^2\cong M$.  Apply Corollary~\ref{cor:square-zero}.
\end{proof}

\end{document}